\documentclass[12pt]{article}
 
 \usepackage{amsmath,amssymb, amsthm}
 
\usepackage{xcolor}
\usepackage{xspace}
\usepackage{virginialake}
\usepackage{adjustbox}

\usepackage{amsfonts}
\usepackage{graphicx}
\newtheorem{thm}{Theorem}[section]

\newtheorem{dfn}[thm]{Definition}

\newcommand{\AND}{\wedge}
\newcommand{\OR}{\vee}
\newcommand{\IMP}{\rightarrow}
\newcommand{\IFF}{\leftrightarrow}

\newcommand{\logicname}[1]{{\sf#1}\xspace}
\newcommand{\f}{\logicname{F}}

\newcommand{\bpc}{\logicname{BPC}}
\newcommand{\ipc}{\logicname{IPC}}

\newcommand{\wf}{\logicname{WF}}

\newcommand{\propset}{\mathsf{Prop}}

\newcommand{\intrule}[1]{#1{I}}
\newcommand{\elrule}[1]{#1{E}}

\begin{document}

\title{Natural Deduction  systems for some weak Subintuitionistic Logics}

\author{\textbf{Fatemeh Shirmohammadzadeh Maleki}\\
Department of Logic, Iranian Institute of Philosophy,\\
 Arakelian 4, Vali-e-Asr, Tehran, Iran,
f.shmaleki2012@yahoo.com}
\maketitle
\date{}
\begin{abstract}
The paper is devoted to the introduction of natural deduction systems for some weak subintuitionistic logics, along with proofs of normalization theorems for these systems.
\end{abstract}


\section{introduction}

The study of subintuitionistic logics began with the work of G. Corsi~\cite{a4}, who introduced a basic system, \f, using a Hilbert-style proof framework. The system \f is characterized by Kripke frames that impose no conditions on the accessibility relation. Corsi also demonstrated that \f can be translated into the modal logic {\sf K}, similar to how \ipc translates into {\sf S4}. G. Restall~\cite{a2} defined a comparable system {\sf SJ}, (see also~\cite{Dic}).

A widely studied extension of \f was introduced by A. Visser, known as Basic Logic \bpc, formulated in natural deduction style. Visser established the completeness of \bpc for finite, irreflexive Kripke models. Substantial research on this logic has since been conducted by M. Ardeshir and W. Ruitenburg (see, for instance, \cite{Ar0, Ar1, Ar}).

D. de Jongh and F. Shirmohammadzadeh Maleki have introduced weaker subintuitionistic logics based on neighborhood semantics~\cite{Dic3, Dic4, FD}. Their work specifically emphasizes the system \wf, which is considerably weaker than \f, along with its extensions. These systems possess modal counterparts, establishing connections between subintuitionistic and modal logics~\cite{Dic4, FD6, Dic5}.

Sequent calculi for \f are explored in \cite{kik, ish, yam, ab, Tesi}. In \cite{Fatemeh}, F. Shirmohammadzadeh Maleki introduced and analyzed sequent calculi for specific weak subintuitionistic logics, including \wf and 
\f, as well as intermediate logics. To date, however, no natural deduction system has been provided for the logics positioned between \wf and \f. This paper aims to bridge this gap.

In this paper, we introduce natural deduction systems for the subintuitionistic logics \wf and \f, as well as for certain logics that lie between these two systems.

The structure of the paper is as follows: Section~\ref{modneigh} provides an overview of Hilbert-style proof systems for selected subintuitionistic logics. In Section~\ref{binneigh}, we introduce natural deduction systems for each of the subintuitionistic logics: \wf, ${\sf WF_N}$, ${\sf WF_{N_{2}}}$, ${\sf WF}{\sf \widehat{C}}$, ${\sf WF}{\sf \widehat{D}}$, ${\sf WFI}$, ${\sf WFC}$, ${\sf WFD}$, and \f. We also demonstrate the equivalence between these natural deduction systems and their corresponding Hilbert systems. In Section~\ref{binneighh}, we discuss the normalization process for the natural deduction systems of \wf and 
\f, along with some intermediate logics.

\section{Preliminaries}\label{modneigh}
In this subsection, we review the Hilbert-style system for the basic subintuitionistic logic {\sf WF} and some of its extensions. The results presented here were previously established in \cite{FD, Dic4}.

The language of subintuitionistic logics is built from a countable set of atomic propositions, denoted by lowercase letters 
$p, q, \dots$ from the set $\propset$, using the connectives $ \vee, \wedge, \rightarrow $ and the propositional constant $ \bot $. We use uppercase latin letters $A, B, C, \ldots$ to represent formulas in this language. The symbol $ \leftrightarrow  $ is defined as an abbreviation by $ A\leftrightarrow B \equiv (A\rightarrow B)\wedge (B\rightarrow A) $ 

\begin{figure}[t]
	\begin{center}
		
	\begin{tabular}{c l @{\hspace{1cm}} cl }
		1. & $ A \IMP (A \OR B )$ & 8. & $A \IMP A$ \\
		2. & $ B \IMP (A \OR B )$ & 9. &  $\vliinf{}{}{A\IMP C}{A \IMP B}{B \IMP C}$\\
		3. & $ (A \AND B) \IMP A$ & 10. & $\vliinf{}{}{A \IMP (B \AND C)}{A \IMP B }{A \IMP C}$\\
		4. & $ (A \AND B) \IMP B$ & 11.  & $\vliinf{}{}{(A \OR B ) \IMP C}{A \IMP C}{B \IMP C}$\\
		5. & $\vliinf{}{}{B}{A}{A \IMP B}$ &12. & $\vliinf{}{}{ A \AND B}{A}{B}$\\
		6. & $\vlinf{}{}{B \IMP A}{A} $& 13. & $\vliinf{}{}{(A \IMP C) \IFF (B \IMP D) }{A \IFF B}{C \IFF D}$\\
		7. & $ A \AND(B \OR C) \IMP (A \AND B ) \OR (A \AND C) $ &14. & $\bot \IMP A$  \\
	\end{tabular}
\end{center}
\caption{The Hilbert-style system for $\wf$ }
\label{fig:axioms:wf}
\end{figure}

 \begin{dfn}\label{def:hilbert:wf} 
The Hilbert-style axiomatization of the basic subintuitionistic logic \wf consists of the axioms and inference rules reported in Figure~\ref{fig:axioms:wf}. 
\end{dfn}  
In Figure~\ref{fig:axioms:wf}, the rules should be applied in such a way that, if the formulas above the line are theorems of \wf, then the formula below the line is a theorem as well and we will refer to rules 5, 6, and 12 as the modus ponens (MP), a fortiori (AF), and conjunction rules, respectively.
The basic notion $ \vdash_{\wf} A$ means that $ A $ can be drived from the axioms of the \wf by means of its rules. But, when one axiomatizes local validity, not all rules in a Hilbert type system have the same status when one considers deductions from assumptions.
 In the case of deductions from assumptions, we impose restrictions on all rules except the conjunction rule.
The restriction on modus ponens is slightly weaker than on the other rules; when concluding $ B $ from $ A $, $ A\rightarrow B $ only the implication $ A\rightarrow B $ need be a theorem. These considerations lead to the following definition.

In this paper we will denote that a formula $ A $ is derivable from $ \Gamma $ in the Hilbert style system of logic {\sf L} as $ \Gamma\,{\vdash_{H{\sf L}}} A $.

\begin{dfn}\label{hh}
We define $ \Gamma\,{\vdash_{H\wf}} A $ iff there is a  derivation of A from $ \Gamma $ using the rules 6, 9, 10 , 11 and 13 of Figure~\ref{fig:axioms:wf}, only when there are  no assumptions, and the rule 5, MP, only when the derivation of $ A\rightarrow B $ contains no assumptions.
\end{dfn}  

The logic \wf can be extended by adding various rules and axiom schemes. In this paper, we will focus on the following axioms and rules:
$$(A\rightarrow B)\wedge (B\rightarrow C)\rightarrow(A\rightarrow  C)~~~~~~~~~~{\sf I}$$
$$(A\rightarrow B)\wedge (A\rightarrow C)\rightarrow(A\rightarrow B\wedge C)  ~~~~{\sf C}$$
$$(A\rightarrow C)\wedge (B\rightarrow C)\rightarrow(A\vee B \rightarrow C)  ~~~~{\sf D}$$
$$(A\rightarrow B\wedge C)\rightarrow (A\rightarrow B)\wedge (A\rightarrow C) ~~~~{\sf \widehat{C}} $$
$$ (A\vee B \rightarrow C)\rightarrow (A\rightarrow C)\wedge (B\rightarrow C) ~~~~{\sf \widehat{D}}$$

$$ \frac{A\rightarrow B\vee C~~~~~C\rightarrow A\vee D~~~~~A\wedge  D\rightarrow B~~~~~ C\wedge B\rightarrow D}{(A\rightarrow B)\leftrightarrow (C\rightarrow D)}~~~~{\sf N}$$
$$~~~~~ \frac{C\rightarrow A\vee D~~~~~~~ C\wedge B\rightarrow D}{(A\rightarrow B)\rightarrow (C\rightarrow D)}~~~~~~~~~~{\sf N_{2}}$$
If $ \Gamma\subseteq\lbrace {\sf I}, {\sf C},{\sf D}, {\sf \widehat{C}}, {\sf \widehat{D}},  {\sf N}, {\sf N_{2}}\rbrace $, we denote by 
${\sf WF\Gamma }$ the logic obtained from {\sf WF} by adding the schemas and rules in 
$ \Gamma $ as new axioms and rules.
Instead of writing ${\sf WFN}$  and ${\sf WFN_{2}}$ we will use ${\sf WF_N}$  and ${\sf WF_{N_{2}}}$, respectively.
The logic \f is the smallest set of formulas closed under instances of \wf, {\sf C}, {\sf D} and {\sf I} \cite{FD}. 

A weak form of the deduction theorem with a single assumption is obtained:

\begin{thm}\label{z3}
{\rm (Weak Deduction theorem, \cite{FD, Dic4})}
\begin{enumerate}
\item[$ \bullet $] $A\vdash_{H{\sf WF\Gamma}} B  ~$ iff $ ~\vdash_{H{\sf WF\Gamma}} A\rightarrow B $.
\item[$ \bullet $] $A_1,\dots,A_n\vdash_{H{\sf WF\Gamma}} B  ~$ iff $~ \vdash_{H{\sf WF\Gamma}} A_1\wedge\dots\wedge A_n\rightarrow B $.
\end{enumerate}
\end{thm}

In the subsequent sections,  we will present the first natural deduction system for the subintuitionistic logics \wf,  ${\sf WF_N}$, ${\sf WF_{N_{2}}}$, ${\sf WF}{\sf \widehat{C}}$, ${\sf WF}{\sf \widehat{D}}$, ${\sf WFI}$, ${\sf WFC}$, ${\sf WFD}$ and \f.

\section{Natural Deduction system}\label{binneigh}
In this section, along with introducing the natural deduction system for the logics \wf,  ${\sf WF_N}$, ${\sf WF_{N_{2}}}$, ${\sf WF}{\sf \widehat{C}}$, ${\sf WF}{\sf \widehat{D}}$, ${\sf WFI}$, ${\sf WFC}$, ${\sf WFD}$ and \f , we will prove that these natural deduction systems are equivalent to their corresponding Hilbert systems.

Assumptions in a natural deduction proof can be either canceled (or discharged) or uncanceled (or open). Initially, all assumptions are open, but specific inferences allow for some assumptions to be discharged. This means that by making such an inference, an assumption that was previously needed is no longer required.
To denote that a formula has been discharged, we will enclose it in square brackets, as in $ [A] $.

In this section we will denote that a formula $ A $ is derivable from $ \Gamma $ in the natural deduction proof system of logic {\sf L} as $ \Gamma  \vdash_{N\!{\sf L}}A$.

\begin{dfn}
We define $ \Gamma  \vdash_{N\!{\sf L}} A$ if there is a derivation with conclusion $ A $ and with all (uncancelled) hypotheses in $\Gamma $. If $\Gamma=\varnothing $, we write $  \vdash_{N\!{\sf L}} A$, and we say that $ A $ is a theorem.
\end{dfn}

\subsection{Natural Deduction system for {\sf WF}}

\begin{dfn}\label{def:natded:wf}
	The rules of the natural deduction system for \wf can be found in Figure~\ref{fig:natded:wf}. 
\end{dfn}

\begin{figure}[]
	\begin{center}
		\begin{tabular}{c c}
			$\vliinf{}{\intrule{\AND}}{ A \AND B}{ 
				\begin{matrix}
					\mathcal{D}_{1} \\
					A
				\end{matrix} 
			}{
				\begin{matrix}
					\mathcal{D}_{2} \\
					B
				\end{matrix} 
			}$ 
			& 
			$\vlinf{i\in\{1,2\}}{\elrule{\AND}_i}{A_i}{
				\begin{matrix}
				\mathcal{D} \\
				A_1 \AND A_2
			\end{matrix}
			}$
			\\[1.5cm]
				$\vlinf{i\in\{1,2\}}{\intrule{\OR}_i}{A_1 \AND A_2}{
				\begin{matrix}
					\mathcal{D} \\
					A_i
				\end{matrix}
			}$
			&
			$
			\vliiinf{}{\elrule{\OR}}{C}{
			\begin{matrix}
				~\\
				\mathcal{D} \\
				A \OR B 
			\end{matrix}
			}{
			\begin{matrix}
				[A]\\
				\mathcal{D}_{1} \\
				C
			\end{matrix}
			}{
			\begin{matrix}
				[B]\\
				\mathcal{D}_{2} \\
				C
			\end{matrix}
			}
			$
			\\[1.5cm]
			$\vlinf{(\star)}{\intrule{\IMP}}{A \IMP B}{
			\begin{matrix}
				[A]\\
				\mathcal{D} \\
				B
			\end{matrix}
			}$
			& 
			$\vliinf{(\ddagger)}{\elrule{\IMP}}{B}{
				\begin{matrix}
					~\\
					\mathcal{D}_{1} \\
					A
				\end{matrix}
			}{ 
				\begin{matrix}
					~\\
					\mathcal{D}_{2} \\
					A \IMP B
				\end{matrix}
			}$
		\\[1.5cm]
		$\vliiinf{(\dagger)}{\intrule{\IMP}_1}{(A \IMP B)\IMP (A\IMP D)}{
			\begin{matrix}
				[B]\\
				\mathcal{D}_{1} \\
				D
			\end{matrix}
		}{\qquad}{
		\begin{matrix}
			[D]\\
			\mathcal{D}_{2} \\
			B
			\end{matrix}
		}$ & 
		$\vliiinf{(\dagger)}{\intrule{\IMP}_2}{(B \IMP A)\IMP (D\IMP A)}{
			\begin{matrix}
				[B]\\
				\mathcal{D}_{1} \\
				D
			\end{matrix}
		}{\qquad}{
			\begin{matrix}
				[D]\\
				\mathcal{D}_{2} \\
				B
			\end{matrix}
		}$
		\\[1.5cm]
		\multicolumn{2}{c}{
			$\vlinf{}{\bot}{A}{
			\begin{matrix}
				\mathcal{D} \\
				\bot 
			\end{matrix}
			}$
		}\\[1cm]
		\multicolumn{2}{l}{{\footnotesize$(\star)$ The rule can be applied only if $A$ is the only assumption.}}\\
        \multicolumn{2}{l}{{\footnotesize$(\ddagger)$ The rule can be applied only if there are no (uncancelled) assumptions for $A\rightarrow B  $.}}\\
		\multicolumn{2}{l}{{\footnotesize$(\dagger)$ The rule can be applied only if $B$ and $D$ are the only assumptions.}}
		\end{tabular}
	\end{center}
\caption{Natural deduction rules for \wf}
\label{fig:natded:wf}
\end{figure}

Most of the  rules of the natural deduction system for \wf are as the natural deduction rules of intuitionistic logic, except for $\rightarrow\!I, \rightarrow\!I_{1}, \rightarrow\!I_{2}$ and $ \rightarrow\!E $.
Note that in the rules $\rightarrow\!I, \rightarrow\!I_{1}, \rightarrow\!I_{2}$ and $ \vee E $ hypotheses cancelled, this is indicated by the striking out of the hypothesis. 
For further details regarding the rule $\rightarrow\!E $, note that all assumptions in the subderivation starting at $ A\rightarrow B $ need to be closed by rules that occur in the subderivation starting at $ A\rightarrow B $. This means that we can apply the rule of $\rightarrow\!E $ only if the derivation having root at $A \rightarrow B$ does not have any other open (that is, not discharged) assumptions.
For example, we can derive the formula  $(A\rightarrow B)\wedge (B\rightarrow C)\rightarrow (A\rightarrow C)$ in the intuitionistic logic (\ipc) as follows:
$$ \frac{\frac{\frac{\frac{\frac{[(A\rightarrow B)\wedge (B\rightarrow C)]^{1}}{A\rightarrow B}\wedge E~~~~\genfrac{}{}{0pt}{}{}{[A]^{2}}}{B}\rightarrow E~~~~\frac{[(A\rightarrow B)\wedge (B\rightarrow C)]^{1}}{B\rightarrow C}\wedge E}{C}\rightarrow E}{A\rightarrow C}\rightarrow I^{2}}{(A\rightarrow B)\wedge (B\rightarrow C)\rightarrow (A\rightarrow C)}\rightarrow I^{1} $$
But according to the  natural deduction system for \wf, we cannot have the above inference in the natural deduction system \wf (we already knew that this axiom is not valid in the \wf), because the assumption $(A\rightarrow B)\wedge (B\rightarrow C) $  in the subderivation starting at $ A\rightarrow B $ did not discharged and so we can not use $\rightarrow E$ rule.
Note that, the assumptions of $ A\rightarrow B $ are closed by the rules of introduction implication, which are outside of the subderivation having $ A\rightarrow B $ as root.

The following theorems show that the Hilbert-style proof system and the natural deduction proof system of subintuitionistic  logic {\sf  WF} are equivalent.

\begin{thm}\label{11}
If $ \Gamma  \vdash_{H\wf}A$ then $  \Gamma\vdash_{N\!\wf} A $.
\end{thm}
\begin{proof}
It is enough to show that all  axioms and rules of  Definition~\ref{def:hilbert:wf} can be deduced in the  natural deduction proof system of {\sf WF}. 
By definitions it is obvious that the rules 5, 6 and 12 of  Definition~\ref{def:hilbert:wf} are the same as the rules $\rightarrow\!E$, $\rightarrow\!I$ and $\wedge I$ of  Definition~\ref{def:natded:wf}, respectively.  So, in the following we just prove that the rules 9, 10, 11 and 13 of  Definition~\ref{def:hilbert:wf} can be deduced in the  natural deduction proof system of {\sf WF}. 

Rule $\textit{9}:~ \frac{A\rightarrow B~~B\rightarrow C}{A\rightarrow C}$

$$
\vlderivation{
\vlin{}{\intrule{\IMP}^1}{A \IMP C}{
\vliin{}{\elrule\IMP}{C}{
\vliin{}{\elrule{\IMP}}{B}{
\vlhy{[A]^1}
}{
\vlhy{A \IMP B}
}	
}{\vlhy{B\IMP C}}
}
}
$$

Rule $\textit{10}:~\frac{A\rightarrow B~~A\rightarrow C}{A\rightarrow B\wedge C}$

$$\vlinf{}{\intrule{\IMP}^{1}}{A \IMP B\wedge C}{
			\begin{matrix}
				\vliinf{}{\intrule{\AND}}{ B \AND C}{ 
				\begin{matrix}
					\vliinf{}{\elrule{\IMP}}{B}{
				\begin{matrix}
					~\\	
				[A]^{1}
				\end{matrix}
			}{ 
				\begin{matrix}
					~\\	
					A \IMP B
				\end{matrix}
			}
				\end{matrix} 
			}{
				\begin{matrix}
					\vliinf{}{\elrule{\IMP}}{C}{
				\begin{matrix}
					~\\	
					[A]^{1}
				\end{matrix}
			}{ 
				\begin{matrix}
					~\\
					A \IMP C
				\end{matrix}
			}
				\end{matrix} 
			}
			\end{matrix}
			}$$

Rule $\textit{11}:~\frac{A\rightarrow C~~B\rightarrow C}{A\vee B \rightarrow  C}$
$$\vlinf{}{\intrule{\IMP}^{3}}{A\OR B \IMP C}{
			\begin{matrix}
				\vliiinf{}{\elrule{\OR}^{1,2}}{C}{
			\begin{matrix}
				~\\
	\\
				[A \OR B]^{3}
			\end{matrix}
			}{
			\begin{matrix}
				\vliinf{}{\elrule{\IMP}}{C}{
				\begin{matrix}
					~\\	
					[A]^{1}
				\end{matrix}
			}{ 
				\begin{matrix}
					~\\	
					A \IMP C
				\end{matrix}
			}
			\end{matrix}
			}{
			\begin{matrix}
				\vliinf{}{\elrule{\IMP}}{C}{
				\begin{matrix}
					~\\
			     [B]^{2}	
				\end{matrix}
			}{ 
				\begin{matrix}
					~\\	
					B \IMP C
				\end{matrix}
			}
			\end{matrix}
			}
			\end{matrix}
			}$$

Rule $\textit{13}:~\frac{A\leftrightarrow B~~C\leftrightarrow D}{(A\rightarrow C)\leftrightarrow (B\rightarrow D)}$
		
$$  \frac{\frac{\frac{\genfrac{}{}{0pt}{}{}{[A\rightarrow C]^{5}}~~~\frac{\frac{[A]^{1}~~~A\rightarrow B}{B}\rightarrow\!E~~~\frac{[B]^{2}~~~B\rightarrow A}{A}\rightarrow\!E}{(A\rightarrow C)\rightarrow (B\rightarrow C)}\rightarrow\!I^{1,2}_{1}}{B\rightarrow C}\rightarrow\!E~~~\frac{\frac{[C]^{3}~~~C\rightarrow D}{D}\rightarrow\!E~~~\frac{[D]^{4}~~~D\rightarrow C}{C}\rightarrow\!E}{(B\rightarrow C)\rightarrow (B\rightarrow D)}\rightarrow\!I^{3,4}_{2}}{B\rightarrow D}\rightarrow\!E}{(A\rightarrow C)\rightarrow (B\rightarrow D)}\rightarrow\!I^{5}$$

We call this derivation $ \mathcal{D} $. So, we have the derivation $ \mathcal{D} $ with conclusion $  (A\rightarrow C)\rightarrow (B\rightarrow D)$.
Similarly to this we can construct a derivation $ \mathcal{D^{'}} $ with conclusion $(B\rightarrow D)\rightarrow (A\rightarrow C)$. Hence:


$$\vliinf{}{\intrule{\AND}}{ (A\rightarrow C)\leftrightarrow (B\rightarrow D)}{ 
				\begin{matrix}
					\mathcal{D} \\
					(A\rightarrow C)\rightarrow (B\rightarrow D)
				\end{matrix} 
			}{
				\begin{matrix}
					 \mathcal{D^{'}}  \\
					(B\rightarrow D)\rightarrow(A\rightarrow C)
				\end{matrix} 
			}$$

Finally, we show that  the axiom 7 of Definition~\ref{def:hilbert:wf} can be deduced in the natural deduction proof system of {\sf WF} as follows. Other axioms are easy to deduce.
$$\frac{\dfrac{\frac{\frac{\frac{[A\wedge (B\vee C)]^{1}}{A}\wedge E~~~~[B]^{2}}{A\wedge B}\wedge I}{(A\wedge B) \vee(A\wedge C)}\vee I~~~~~\frac{\frac{\frac{[A\wedge (B\vee C)]^{1}}{A}\wedge E~~~~[C]^{3}}{A\wedge C}\wedge I}{(A\wedge B) \vee(A\wedge C)}\vee I~~~~\frac{[A\wedge (B\vee C)]^{1}}{B\vee C}(\wedge E)}{(A\wedge B) \vee(A\wedge C)}\vee E^{2,3}}{A\wedge (B\vee C)\rightarrow (A\wedge B) \vee(A\wedge C)} \rightarrow I^{1}$$
\end{proof}

\begin{thm}\label{12}
If $ \Gamma  \vdash_{N\!\wf}A$ then $  \Gamma\vdash_{H\wf} A $.
\end{thm}

\begin{proof}
The proof is by induction on the complexity of $ A $ from $\Gamma  $:

Assume $ A \in  \Gamma  $, then it is obvious that $  \Gamma\vdash_{H\wf} A $.

Assume $ A=B\wedge C $, $ \Gamma  \vdash_{N\wf}A$ and the last derivation is by the rule $\wedge I$. Then we have $ \Gamma  \vdash_{N\wf}B$ and $ \Gamma  \vdash_{N\wf}C$. So by the induction hypothsis we conclude $ \Gamma  \vdash_{H\wf}B$ and $ \Gamma  \vdash_{H\wf}C$. Hence by the conjunction rule we conclude that $ \Gamma  \vdash_{H\wf} B\wedge C$.

Assume $ \Gamma  \vdash_{N\wf}A$ and the last derivation is by the rule $\wedge E$. Then there exists a formula $ B $ such that  $ \Gamma  \vdash_{N\wf}A\wedge B$. So, by the induction hypothesis $ \Gamma  \vdash_{H\wf}A\wedge B$. On the other hand we have $  \vdash_{H\wf}A\wedge B\rightarrow A  $. Hence, by the MP rule we conclude that $ \Gamma  \vdash_{H\wf} A$.

Assume  $ A=B\rightarrow C $ and the last derivation is by the rule $\rightarrow I$. This means that $B \vdash_{N\wf}   C$. Then by the induction hypothesis, $B \vdash_{H\wf}   C$. Hence by Theorem~\ref{z3}, we conclude $  \vdash_{H\wf} B\rightarrow   C $ and so $ \Gamma  \vdash_{H\wf} B\rightarrow C$.

Assume $ A=(C\rightarrow B)\rightarrow (C\rightarrow D) $, $ \Gamma  \vdash_{N\wf}A$ and the last derivation is by the rule $ \rightarrow\!I_{1} $. This implies that $ D\vdash_{N\wf} B$ and  $B\vdash_{N\wf} D$. So, by the induction hypothesis we have  $ D\vdash_{H\wf} B$ and  $B\vdash_{H\wf} D$. By applying Theorem~\ref{z3} and the conjunction rule, we conclude $ \vdash_{H\wf} B\leftrightarrow D  $. On the other hand, we have $ \vdash_{H\wf} C\leftrightarrow C  $. Therefore, by rule 13, we have  $\Gamma  \vdash_{H\wf}  (C\rightarrow B)\leftrightarrow (C\rightarrow D) $, and hence $ \Gamma  \vdash_{H\wf}A$.

The proofs of the other cases are similar.
\end{proof}

\subsection{Natural Deduction system for ${\sf WF_{N}}$}
If we consider the rule $ \rightarrow_{N}\!I $ instead of $ \rightarrow\!I_{1} $ and $ \rightarrow\!I_{2} $  in the natural deduction system for \wf, then we obtain the natural deduction system for ${\sf WF_{N}}$.
$$	\vlderivation {
\vliiiin{}{\rightarrow_{N}\!I}{(A\rightarrow B)\rightarrow (C\rightarrow D)}{
\vlhy {\begin{matrix}
				[A]\\
				\mathcal{D}_{0} \\
				C\vee B
			\end{matrix}}}
{
\vlhy {\begin{matrix}
				[D]\\
				\mathcal{D}_{1} \\
				B
			\end{matrix}}}
{
\vlhy {\begin{matrix}
				[C]\\
				\mathcal{D}_{2} \\
				A\vee D
			\end{matrix}}}
{
\vlhy {\begin{matrix}
				[B]\\
				\mathcal{D}_{3} \\
				D
			\end{matrix}}}}$$
	
Note that the rule $ \rightarrow_{N}\!I $ can be applied only if $A$, $ B $, $ C $ and $ D $ are the only assumptions.

The following theorems show that the Hilbert-style proof system and the natural deduction proof system of subintuitionistic  logic ${\sf WF_{N}}$ are equivalent.

\begin{thm}\label{22}
If $ \Gamma  \vdash_{H{\sf WF_{N}}}A$ then $  \Gamma\vdash_{N\!{\sf WF_{N}}} A $.
\end{thm}
\begin{proof}
By Theorem \ref{11}, we only need to show that the rule $ {\sf N} $  can be deduced in the  natural deduction proof system of ${\sf WF_{N}}$:
$$  \frac{\frac{[A]^{1}~~A\rightarrow C\vee B}{C\vee B}\!\rightarrow\!E~~\frac{\frac{[D]^{2}~~[A]^{1}}{D\wedge A}\!\wedge\! I~~\genfrac{}{}{0pt}{}{ }{D\wedge A\rightarrow B}}{B}\!\rightarrow\!E~~\frac{[C]^{3}~~C\rightarrow A\vee D}{A\vee D}\!\rightarrow\!E~~\frac{\frac{[B]^{4}~~[C]^{3}}{B\wedge C}\!\wedge\!I~~\genfrac{}{}{0pt}{}{}{B\wedge C\rightarrow D}}{ D}\rightarrow\!E}{(A\rightarrow B)\rightarrow (C\rightarrow D)}\!\rightarrow_{N}\!I$$
and
$$  \frac{\frac{[C]^{3}~~C\rightarrow A\vee D}{A\vee D}\!\rightarrow\!E~~\frac{\frac{[B]^{4}~~[C]^{3}}{B\wedge C}\!\wedge\!I~~\genfrac{}{}{0pt}{}{}{B\wedge C\rightarrow D}}{ D}\rightarrow\!E  ~~\frac{[A]^{1}~~A\rightarrow C\vee B}{C\vee B}\!\rightarrow\!E~~\frac{\frac{[D]^{2}~~[A]^{1}}{D\wedge A}\!\wedge\! I~~\genfrac{}{}{0pt}{}{ }{D\wedge A\rightarrow B}}{B}\!\rightarrow\!E}{(C\rightarrow D)\rightarrow (A\rightarrow B)}\!\rightarrow_{N}\!I$$
\end{proof}

\begin{thm}\label{33}
If $ \Gamma  \vdash_{N\!{\sf WF_{N}}} A$ then $  \Gamma\vdash_{H{\sf WF_{N}}} A $.
\end{thm}

\begin{proof}
The proof is by induction on the complexity of $ A $ from $\Gamma  $. By Theorem \ref{12}, we only need to prove the case where the last derivation is by the rule $ \rightarrow_{N}\!I$. 

Assume $ A=(E\rightarrow B)\rightarrow (C\rightarrow D) $, $ \Gamma  \vdash_{N\wf}A$ and the last derivation is by the rule $\rightarrow_{N}\!I$. This implies that $ E\vdash_{N\wf} C\vee B$, $D\vdash_{N\wf} B$, $ C\vdash_{N\wf} E\vee D $ and $ B\vdash_{N\wf} D $. Thus, by the induction hypothesis we have $E\vdash_{H{\sf WF_{N}}} C\vee B$, $D\vdash_{H{\sf WF_{N}}} B$, $ C\vdash_{H{\sf WF_{N}}} E\vee D $ and $ B\vdash_{H{\sf WF_{N\wf}}} D $. By appling Theorem~\ref{z3}, we conclude $ \vdash_{H{\sf WF_{N}}} E\rightarrow  C\vee B$, $ \vdash_{H{\sf WF_{N}}} D\rightarrow  B$, $ \vdash_{H{\sf WF_{N}}} C\rightarrow  E\vee D$  and  $ \vdash_{H{\sf WF_{N}}} B\rightarrow  D$. Therefor, we conclude that $ \vdash_{H{\sf WF_{N}}} C\wedge B\rightarrow  D$ and $ \vdash_{H{\sf WF_{N}}} E\wedge D\rightarrow  B$. Hence by rule $ N $, we conclude that $\Gamma  \vdash_{H{\sf WF_{N}}}  (E\rightarrow B)\leftrightarrow (C\rightarrow D) $. Therefore, $ \Gamma  \vdash_{H{\sf WF_{N}}}A$.

\end{proof}

\subsection{Natural Deduction system for ${\sf WF_{N_{2}}}$}

If we consider the rule $\rightarrow_{N_{2}}\!I$  instead of $ \rightarrow\!I_{1} $ and $ \rightarrow\!I_{2} $ in the natural deduction system for \wf, then we obtain the natural deduction system for ${\sf WF_{N_{2}}}$.
$$\vlderivation {
\vliin{}{\rightarrow_{N_{2}}\!I}{(A\rightarrow B)\rightarrow (C\rightarrow D)}{
\vlhy {{
		\begin{matrix}
				[C]\\
			\mathcal{D}_{0} \\
			A\vee D
		\end{matrix}
		}}}
{
\vlhy {{
		\begin{matrix}
				[B]\\
			\mathcal{D}_{1} \\
			D
		\end{matrix}
		}}}} $$
Note that the rule $ \rightarrow_{N_{2}}\!I $ can be applied only if $ C $ and $ B $ are the only assumptions.

The following theorems show that the Hilbert-style proof system and the natural deduction proof system of subintuitionistic  logic ${\sf WF_{N_{2}}}$ are equivalent.

\begin{thm}
If $ \Gamma  \vdash_{H{\sf WF_{N_{2}}}}A$ then $  \Gamma\vdash_{N\!{\sf WF_{N_{2}}}} A $.
\end{thm}
\begin{proof}
The proof is similar to the proof of Theorem \ref{22}.
\end{proof}

\begin{thm}
If $ \Gamma  \vdash_{N\!{\sf WF_{N_{2}}}} A$ then $  \Gamma\vdash_{H{\sf WF_{N}}} A $.
\end{thm}
\begin{proof}
The proof is similar to the proof of Theorem \ref{33}.
\end{proof}

\subsection{Natural Deduction system for ${\sf WF\hat{C}}$}

If we add the following rule to the natural deduction system for \wf, we obtain the natural deduction system for ${\sf WF\hat{C}}$:

$$\vlinf{}{\rightarrow_{\hat{\wedge}}\!I}{(C\rightarrow A) \IMP (C\rightarrow B)}{
			\begin{matrix}
				[A]\\
			\mathcal{D} \\
				B
			\end{matrix}
			}$$
Note that the rule $\rightarrow_{\hat{\wedge}}\!I$ can be applied only if $A$ is the only assumption.

The following theorems show that the Hilbert-style proof system and the natural deduction proof system of subintuitionistic  logic ${\sf WF\hat{C}}$ are equivalent.

\begin{thm}\label{44}
If $ \Gamma  \vdash_{H{\sf WF\hat{C}}}A$ then $  \Gamma\vdash_{N\!{\sf WF\hat{C}}} A $.
\end{thm}
\begin{proof}
By Theorem \ref{11}, we only need to show that the axiom $\hat{C}$  can be deduced in the  natural deduction proof system of ${\sf WF\hat{C}}$:
$$\frac{\frac{\frac{\frac{\frac{[B\wedge C]^{1}}{B}\wedge E}{(A\rightarrow B\wedge C)\rightarrow (A\rightarrow B)}\rightarrow_{\hat{\wedge}}\!I^{1}~~~\genfrac{}{}{0pt}{}{}{[A\rightarrow B\wedge C]^{2}}}{A\rightarrow B}\rightarrow E~~~\frac{\frac{\frac{[B\wedge C]^{1}}{C}\wedge E}{(A\rightarrow B\wedge C)\rightarrow (A\rightarrow C)}\rightarrow_{\hat{\wedge}}\!I^{1}~~~\genfrac{}{}{0pt}{}{}{[A\rightarrow B\wedge C]^{2}}}{A\rightarrow C}\rightarrow E}{ (A\rightarrow B)\wedge (A\rightarrow C)}\wedge I}{(A\rightarrow B\wedge C)\rightarrow (A\rightarrow B)\wedge (A\rightarrow C)}\rightarrow I^{2} $$
\end{proof}

\begin{thm}\label{444}
If $ \Gamma  \vdash_{N{\sf WF\hat{C}}} A$ then $  \Gamma\vdash_{H{\sf WF\hat{C}}} A $.
\end{thm}

\begin{proof}
The proof is by induction on the complexity of $A$ from $\Gamma  $. By Theorem \ref{12}, we only need to prove the case where the last derivation is by the rule $ \rightarrow_{\hat{\wedge}}\!I$. 

Assume $ A=(C\rightarrow E)\rightarrow (C\rightarrow B) $, $ \Gamma  \vdash_{N{\sf WF\hat{C}}}A$ and the last derivation is by the rule $\rightarrow_{\hat{\wedge}}\!I$. This implies that $ E\vdash_{N{\sf WF\hat{C}}} B$. So, by the induction hypothesis we have $E\vdash_{H{\sf WF\hat{C}}} B $. By applying Theorem~\ref{z3}, we conclude $ \vdash_{H{\sf WF\hat{C}}} E\rightarrow   B$. Therefore:

1. $ \vdash_{H{\sf WF\hat{C}}}    E\rightarrow   B$~~~~~~~~~~~~~~~~~~~~~~~~~~~~~~~~~~~~~~~~~~~~Assumption

2. $ \vdash_{H{\sf WF\hat{C}}}    E\leftrightarrow  E\wedge B $

3. $ \vdash_{H{\sf WF\hat{C}}}  C\leftrightarrow C $

4. $ \vdash_{h{\sf WF\hat{C}}} (C\rightarrow E)\leftrightarrow (C\rightarrow E\wedge B) $~~~~~~~~~~~~~~~~~~~~By 2, 3 using rule 13 in \wf

5. $ \vdash_{H{\sf WF\hat{C}}} (C\rightarrow E\wedge B)\rightarrow (C\rightarrow E)\wedge   (C\rightarrow B)$~~~~The axiom $ \hat{C} $

6. $ \vdash_{h{\sf WF\hat{C}}}  (C\rightarrow E)\rightarrow (C\rightarrow E)\wedge   (C\rightarrow B)$

7. $ \vdash_{H{\sf WF\hat{C}}} (C\rightarrow E)\wedge   (C\rightarrow B)\rightarrow (C\rightarrow B) $

8. $ \vdash_{H{\sf WF\hat{C}}} (C\rightarrow E)\rightarrow (C\rightarrow B) $

Hence, we conclude that $\Gamma  \vdash_{H{\sf WF\hat{C}}} A$.
\end{proof}

\subsection{Natural Deduction system for ${\sf WF\hat{D}}$}
If we add the following rule to the natural deduction system for \wf, we obtain  the natural deduction system for ${\sf WF\hat{D}}$:
$$\vlinf{}{\rightarrow_{\hat{\vee}}\!I}{(B\rightarrow C) \IMP (A\rightarrow C)}{
			\begin{matrix}
				[A]\\
					\mathcal{D} \\
				B
			\end{matrix}
			}$$

Note that  the rule $\rightarrow_{\hat{\vee}}\!I$ can be applied only if $A$ is the only assumption.

The following theorems show that the Hilbert-style proof system and the natural deduction proof system of subintuitionistic  logic ${\sf WF\hat{D}}$ are equivalent.

\begin{thm}
If $ \Gamma  \vdash_{H{\sf WF\hat{D}}}A$ then $  \Gamma\vdash_{N\!{\sf WF\hat{D}}} A $.
\end{thm}
\begin{proof}
The proof is similar to the proof of Theorem \ref{44}.
\end{proof}

\begin{thm}
If $ \Gamma  \vdash_{N\!{\sf WF\hat{D}}} A$ then $  \Gamma\vdash_{H{\sf WF\hat{D}}} A $.
\end{thm}

\begin{proof}
The proof is similar to the proof of Theorem \ref{444}.
\end{proof}

\subsection{Natural Deduction system for ${\sf WFC}$}

If we add the following rule to the natural deduction system for \wf, we obtain the natural deduction system for ${\sf WFC}$:
$$ \vlderivation {
\vliin{}{\rightarrow_{\wedge}\!I}{A\rightarrow B\wedge C}{
\vlhy {{
			\begin{matrix}	
					\mathcal{D}_{0} \\
				A\rightarrow B
			\end{matrix}
			}}}
{
\vlhy {{
			\begin{matrix}	
					\mathcal{D}_{1} \\
			A\rightarrow C
			\end{matrix}
			}}}}$$

The following theorems show that the Hilbert-style proof system and the natural deduction proof system of subintuitionistic  logic ${\sf WFC}$ are equivalent.

\begin{thm}\label{34}
If $ \Gamma  \vdash_{H{\sf WFC}}A$ then $  \Gamma\vdash_{N\!{\sf WFC}} A $.
\end{thm}
\begin{proof}
By Theorem \ref{11}, we only need to show that the axiom {\sf C}  can be deduced in the  natural deduction proof system of ${\sf WFC}$:
$$\frac{\frac{\frac{[(A\rightarrow B)\wedge (A\rightarrow C)]^{1}}{A\rightarrow B}\wedge E~~~~\frac{[(A\rightarrow B)\wedge (A\rightarrow C)]^{1}}{A\rightarrow C}\wedge E}{A\rightarrow B\wedge C}\rightarrow_{\wedge}\!I}{(A\rightarrow B)\wedge (A\rightarrow C)\rightarrow (A\rightarrow B\wedge C)}\rightarrow I^{1} $$
\end{proof}

\begin{thm}\label{35}
If $ \Gamma  \vdash_{N\!{\sf WFC}} A$ then $  \Gamma\vdash_{H{\sf WFC}} A $.
\end{thm}

\begin{proof}
The proof is by induction on the complexity of $ A $ from $\Gamma  $. By Theorem \ref{12}, we only need to prove the case where the last derivation is by the rule $\rightarrow_{\wedge}\!I$. 

Assume $ A=E\rightarrow B\wedge C$, $ \Gamma  \vdash_{N}A$ and the last derivation is by the rule $\rightarrow_{\wedge}\!I$. This implies that $ \Gamma\vdash_{N{\sf WFC}} E\rightarrow B$ and  $ \Gamma\vdash_{N{\sf WFC}} E\rightarrow C$. So, by the induction hypothesis we have $\Gamma\vdash_{H{\sf WFC}} A\rightarrow B $ and $\Gamma\vdash_{H{\sf WFC}} A\rightarrow C$. Then we conclude that $ \Gamma\vdash_{H{\sf WFC}} (E\rightarrow   B)\wedge  (E\rightarrow C)$. By axiom {\sf C} we have $ \vdash_{H{\sf WFC}}(E\rightarrow   B)\wedge  (E\rightarrow C)\rightarrow ( E\rightarrow B\wedge C)$.  Hence, by the MP rule, we conclude $ \Gamma\vdash_{H{\sf WFC}} E\rightarrow B\wedge C $. Therefore, $\Gamma\vdash_{H{\sf WFC}}  A $.
\end{proof}

\subsection{Natural Deduction system for ${\sf WFD}$}

If we add the following rule to the natural deduction system for \wf, we obtain the natural deduction system for ${\sf WFD}$:
$$ \vlderivation {
\vliin{}{\rightarrow_{\vee}\!I}{A\vee B\rightarrow  C}{
\vlhy {{
			\begin{matrix}	
					\mathcal{D}_{0} \\
				A\rightarrow C
			\end{matrix}
			}}}
{
\vlhy {{
			\begin{matrix}	
					\mathcal{D}_{1} \\
			B\rightarrow C
			\end{matrix}
			}}}}$$
The following theorems show that the Hilbert-style proof system and the natural deduction proof system of subintuitionistic  logic ${\sf WFD}$ are equivalent.
\begin{thm}
If $ \Gamma  \vdash_{H{\sf WFD}}A$ then $  \Gamma\vdash_{N\!{\sf WFD}} A $.
\end{thm}
\begin{proof}
The proof is similar to the proof of Theorem \ref{34}.
\end{proof}

\begin{thm}
If $ \Gamma  \vdash_{N\!{\sf WFC}} A$ then $  \Gamma\vdash_{H{\sf WFD}} A $.
\end{thm}
\begin{proof}
The proof is similar to the proof of Theorem \ref{35}.
\end{proof}

\subsection{Natural Deduction system for ${\sf WFI}$}

If we consider the rule $\rightarrow_{tr}\!I$  instead of $ \rightarrow I_{1} $ and $ \rightarrow I_{2} $ in the natural deduction system for \wf, then we obtain the natural deduction system for ${\sf WFI}$.
$$ \vlderivation {
\vliin{}{\rightarrow_{tr}\!I}{A\rightarrow  C}{
\vlhy {{
			\begin{matrix}	
					\mathcal{D}_{0} \\
				A\rightarrow B
			\end{matrix}
			}}}
{
\vlhy {{
			\begin{matrix}	
					\mathcal{D}_{1} \\
			B\rightarrow C
			\end{matrix}
			}}}}$$
The following theorems show that the Hilbert-style proof system and the natural deduction proof system of subintuitionistic  logic ${\sf WFI}$ are equivalent.
\begin{thm}
If $ \Gamma  \vdash_{H{\sf WFI}}A$ then $  \Gamma\vdash_{N\!{\sf WFI}} A $.
\end{thm}
\begin{proof}
The proof is similar to the proof of Theorem \ref{34}.
\end{proof}

\begin{thm}
If $ \Gamma  \vdash_{N\!{\sf WFI}} A$ then $  \Gamma\vdash_{H{\sf WFI}} A$.
\end{thm}
\begin{proof}
The proof is similar to the proof of Theorem \ref{35}.
\end{proof}

\subsection{Natural Deduction system for ${\sf F}$}

If we add the rules $ \rightarrow_{\wedge}\!I $ and $ \rightarrow_{\vee}\!I $ to the natural deduction system of $ {\sf WFI} $, we obtain the natural deduction system for ${\sf F}$.
 
 \begin{thm}
 $ \Gamma  \vdash_{N\!\f} E$ iff $  \Gamma\vdash_{H{\sf F}} E $.
\end{thm}
\begin{proof}
Since the logic \f is the smallest set of formulas closed under instances of \wf, {\sf I}, {\sf C} and {\sf D} \cite{FD}, the proof of this theorem follows from the equivalence of the Hilbert-style proof systems and the natural deduction proof systems in subintuitionistic logics 
${\sf WFI}$, ${\sf WFC}$ and ${\sf WFD}$, which was previously established in earlier sections.
\end{proof}

 \section{Normalization}\label{binneighh}
In this section, we study the process of normalization for the natural deduction systems of the subintuitionistic logics that we introduced in the previous sections.

In the elimination rules, we differentiate between two types of premises: major and minor. In the conjunction elimination rule, the sole premise is considered major. For implication elimination, the premise 
$ A\rightarrow B $ is major, while $ A $ is minor. In the disjunction elimination rule, the disjunction $ A\vee B $, which is being eliminated, acts as the major premise, while the two premises identical to the conclusion are minor.
\begin{dfn}
A derivation is \textbf{normal} if every major premise of an elimination rule is either an assumption or the conclusion of an elimination rule different from the $( \elrule{\OR} )$ rule.
\end{dfn}

\begin{dfn}
A \textbf{cut} is a formula occurrence which is the major premise of an elimination rule and the conclusion of a $( \elrule{\OR} )$ rule or an introduction rule.
\end{dfn}

\begin{dfn}
The logical depth $dp(A)$ of a formula is defined inductively as follows:
the logical depth of a propositional variable or $ \bot $ is $  0$, while the logical depth of $ A\circ B $ is
$max(dp(A), dp(B)) + 1$. The rank $rk(A)$ of a formula $A$ will be defined as $dp(A) + 1$.

If $  \mathcal{D}$ is a derivation, then we define its cut rank to be $  0$, if it contains no cut formulas
(i.e., is cut free). If, on the other hand, $  \mathcal{D}$  does contain cuts, then we define its cut rank to
be the maximum of all the ranks of cut formulas in $  \mathcal{D}$. 
\end{dfn}

\begin{dfn}
A $ segment $ of length $ n $ in a derivation $\mathcal{D}$ is a sequence of formula occurrences $ \langle A_{1}, \dots,   A_{n}\rangle$ of the same formula $ A $ in $\mathcal{D}$ such that:

1. $ A_{1} $ is not the conclusion of an application of the $( \elrule{\OR} )$ rule,

2. $ A_{n} $ is not the minor premiser of the $( \elrule{\OR} )$ rule,

3. every $ A_{i} $, with $ i \lneqq n $ is a minor premise of a $( \elrule{\OR} )$ rule with conclusion $ A_{i+1} $.
\end{dfn}

The following theorems prove the normalization theorem for each of the natural deduction systems of the subintuitionistic logics that we introduced in the previous sections.

\begin{thm}\label{normal}
There is an effective procedure for transforming a natural deduction derivation of $ \Gamma\vdash A $ in subintuitionistic logic \wf into a normal one also showing  $ \Gamma\vdash A $.
\end{thm}
\begin{proof}
Assume we have a derivation with cut rank $ d $ in which all major premises in elimination rules occur to the left of the minor premises. Let $ l $ be the sum of the lengths of all the segments containing a cut formula of rank $ d $. We look at the rightmost cut formula of rank $ d $. All cases are treated as for intuitionistic logic (see, \cite{Tr}), except  for the case where we introduce an implication by $ \intrule{\IMP}_i $ which is then eliminated by $ \elrule{\IMP} $. In this case we are looking at a situation like this:
			
		$$\vliinf{}{\elrule{\IMP}}{A\IMP D}{
				\begin{matrix}
					~\\
					\vliiinf{}{\intrule{\IMP}_1}{(A \IMP B)\IMP (A\IMP D)}{
			\begin{matrix}
				[B]\\
				\mathcal{D}_{0} \\
				D
			\end{matrix}
		}{\qquad}{
		\begin{matrix}
			[D]\\
			\mathcal{D}_{1} \\
			B
			\end{matrix}
		}
				\end{matrix}
			}{ 
				\begin{matrix}
					~\\
					\vlinf{}{\intrule{\IMP}}{A \IMP B}{
			\begin{matrix}
				[A]\\
				\mathcal{D}_{2} \\
				B
			\end{matrix}
			}
				\end{matrix}
			}$$

where the cut rank of $  \mathcal{D}_{2}  $ is strictly smaller than $ d $. 

So, we replace this by:
	
$$\vlinf{}{\intrule{\IMP}}{A \IMP D}{
			\begin{matrix}
				[A]\\
			\mathcal{D}_{2} \\
				B\\
				\mathcal{D}_{0}\\
				D
			\end{matrix}
			}$$
In this case we reduce the number of cut formulas with rank $ d $. Hence, we conclude that after repeated application of these steps we end up with a normal derivation.
\end{proof}

 \begin{thm}\label{normm}
There is an effective procedure for transforming a natural deduction derivation of $ \Gamma\vdash A $ in subintuitionistic logic $ {\sf WF_{N}} $ into a normal one also showing  $ \Gamma\vdash A $.
\end{thm}

\begin{proof}
By the proof of Theorem \ref{normal}, we only need to consider the case where we introduce an implication by $\rightarrow_{N}\!I$ which is then eliminated by $ \elrule{\IMP} $. In this case we are looking at a situation like this:
	
	$$ \vlderivation {
\vliin{}{\rightarrow E}{C\rightarrow D}{
\vlhy {\vlderivation {
\vliiiin{}{\rightarrow_{N}\!I}{(A\rightarrow B)\rightarrow (C\rightarrow D)}{
\vlhy {{
		\begin{matrix}
				[A]\\
			\mathcal{D}_{0} \\
			C\vee B
		\end{matrix}
		}}}
{
\vlhy {{
		\begin{matrix}
				[D]\\
			\mathcal{D}_{1} \\
			B
		\end{matrix}
		}}}
{
\vlhy {{
		\begin{matrix}
				[C]\\
			\mathcal{D}_{2} \\
			A\vee D
		\end{matrix}
		}}}
{
\vlhy {{
		\begin{matrix}
				[B]\\
			\mathcal{D}_{3} \\
			D
		\end{matrix}
		}}}}}}
{
\vlhy {\vlderivation {
\vlin{}{\rightarrow I}{A\rightarrow B}{
\vlhy {{
		\begin{matrix}
				[A]\\
			\mathcal{D}_{4} \\
			B
		\end{matrix}
		}}}}}}}$$

	where the cut rank of $  \mathcal{D}_{4}  $ is strictly smaller than $ d $. So, we replace this by:
	
	$$ \vlderivation {
\vlin{}{\rightarrow I}{C\rightarrow D}{
\vlhy {\vlderivation {
\vliiin{}{\vee E}{D}{
\vlhy {{
		\begin{matrix}
	\\
	\\
				[C]\\
			\mathcal{D}_{2} \\
			A\vee D
		\end{matrix}
		}}}
{
\vlhy {{
		\begin{matrix}
				[A]\\
			\mathcal{D}_{4} \\
			B\\
			\mathcal{D}_{3} \\
			D
		\end{matrix}
		}}}
{
\vlhy {{
		\begin{matrix}
	\\
	\\
				[D]\\
			\mathcal{D} \\
			D
		\end{matrix}
		}}}}}}}$$
In this case we reduce the number of cut formulas with rank $ d $. Hence, we conclude that after repeated application of these steps we end up with a normal derivation.
\end{proof}

\begin{thm}
There is an effective procedure for transforming a natural deduction derivation of $ \Gamma\vdash A $ in subintuitionistic logic $ {\sf WF_{N_{2}}} $ into a normal one also showing  $ \Gamma\vdash A $.
\end{thm}

\begin{proof}
The proof is similar to the proof of Theorem \ref{normm}.
\end{proof}

\begin{thm}\label{normmm}
There is an effective procedure for transforming a natural deduction derivation of $ \Gamma\vdash A $ in subintuitionistic logic ${\sf WF\hat{C}}$ into a normal one also showing  $ \Gamma\vdash A $.
\end{thm}

\begin{proof}
By the proof of Theorem \ref{normal}, we only need to consider the case where we introduce an implication by $\rightarrow_{\hat{\wedge}}\!I$ which is then eliminated by $ \elrule{\IMP} $. In this case we are looking at a situation like this:

$$\vlderivation {
\vliin{}{\rightarrow E}{C\rightarrow B}{
\vlhy {\vlderivation {
\vlin{}{\rightarrow_{\hat{\wedge}}\!I}{(C\rightarrow A)\rightarrow (C\rightarrow B)}{
\vlhy {{
			\begin{matrix}
				[A]\\
			\mathcal{D}_{0} \\
				B
			\end{matrix}
			}}}}}}
{
\vlhy {\vlderivation {
\vlin{}{\rightarrow I}{C\rightarrow A}{
\vlhy {{
			\begin{matrix}
				[C]\\
			\mathcal{D}_{1} \\
				A
			\end{matrix}
			}}}}}}} $$
where the cut rank of $  \mathcal{D}_{1}  $ is strictly smaller than $ d $. So, we replace this by:
$$\vlderivation {
\vlin{}{\rightarrow I}{C\rightarrow B}{
\vlhy {{
			\begin{matrix}
				[C]\\
			\mathcal{D}_{1} \\
				A\\
				\mathcal{D}_{0} \\
				B
			\end{matrix}
			}}}} $$
In this case we reduce the number of cut formulas with rank $ d $. Hence, we conclude that after repeated application of these steps we end up with a normal derivation.
\end{proof}
\begin{thm}
There is an effective procedure for transforming a natural deduction derivation of $ \Gamma\vdash A $ in subintuitionistic logic ${\sf WF\hat{D}}$ into a normal one also showing  $ \Gamma\vdash A $.
\end{thm}

\begin{proof}
The proof is similar to the proof of Theorem \ref{normmm}.
\end{proof}

\begin{thm}
There is an effective procedure for transforming a natural deduction derivation of $ \Gamma\vdash A $ in subintuitionistic logic ${\sf WFC}$ into a normal one also showing  $ \Gamma\vdash A $.
\end{thm}

\begin{proof}
By the proof of Theorem \ref{normal}, we only need to consider the case where we introduce a conjunction by $\rightarrow_{\wedge}\!I$ which is then eliminated by $ \elrule{\IMP} $. In this case we are looking at a situation like this:
$$\vlderivation {
\vliin{}{\rightarrow E}{B\wedge C}{
\vlhy {\vlderivation {
\vliin{}{\rightarrow_{\wedge}\!I}{A\rightarrow B\wedge C}{
\vlhy {{
			\begin{matrix}	
					\mathcal{D}_{0} \\
				A\rightarrow B
			\end{matrix}
			}}}
{
\vlhy {{
			\begin{matrix}	
					\mathcal{D}_{1} \\
			A\rightarrow C
			\end{matrix}
			}}}}}}
{
\vlhy {{
				\begin{matrix}
	\\
					 \mathcal{D}_{2}  \\
					A
				\end{matrix} 
			}}}} $$
where the cut rank of $  \mathcal{D}_{2}  $ is strictly smaller than $ d $. 
Since the rule of $ \rightarrow E $ is applied in this derivation, we conclude, according to the natural deduction for \wf , that all assumptions in the derivations $  	\mathcal{D}_{0}$ and $\mathcal{D}_{1}  $ discharged.
So, we replace this by:

$$\vlderivation {
\vliin{}{\wedge I}{B\wedge C}{
\vlhy {\vlderivation {
\vliin{}{\rightarrow E}{B}{
\vlhy {{
				\begin{matrix}
	\\
					 \mathcal{D}_{2}  \\
					A
				\end{matrix} 
			}}}
{
\vlhy {{
			\begin{matrix}	
	\\
					\mathcal{D}_{0} \\
				A\rightarrow B
			\end{matrix}
			}}}}}}
{
\vlhy {\vlderivation {
\vliin{}{\rightarrow E}{C}{
\vlhy {{
				\begin{matrix}
	\\
					 \mathcal{D}_{2}  \\
					A
				\end{matrix} 
			}}}
{
\vlhy {{
			\begin{matrix}	
	\\
					\mathcal{D}_{1} \\
				A\rightarrow C
			\end{matrix}
			}}}}}}} $$
In this case we reduce the number of cut formulas with rank $ d $. Hence, we conclude that after repeated application of these steps we end up with a normal derivation.
\end{proof}

\begin{thm}
There is an effective procedure for transforming a natural deduction derivation of $ \Gamma\vdash A $ in subintuitionistic logic ${\sf WFD}$ into a normal one also showing  $ \Gamma\vdash A $.
\end{thm}

\begin{proof}
By the proof of Theorem \ref{normal}, we only need to consider the case where we introduce a disjunction by $\rightarrow_{\vee}\!I$ which is then eliminated by $ \elrule{\IMP} $. In this case we are looking at a situation like this ($ i=1, 2 $):
$$\vlderivation {
\vliin{}{\rightarrow E}{ C}{
\vlhy {\vlderivation {
\vliin{}{\rightarrow_{\vee}\!I}{A_{1}\vee A_{2}\rightarrow C}{
\vlhy {{
			\begin{matrix}	
					\mathcal{D}_{1} \\
				A_{1}\rightarrow C
			\end{matrix}
			}}}
{
\vlhy {{
			\begin{matrix}	
					\mathcal{D}_{2} \\
			A_{2}\rightarrow C
			\end{matrix}
			}}}}}}
{
\vlhy {\vlderivation {
\vlin{}{\vee I}{A_{1}\vee A_{2}}{
\vlhy {{
			\begin{matrix}	
					\mathcal{D} \\
				A_{i}
			\end{matrix}
			}}}}}}} $$
where the cut rank of $  \mathcal{D}$ is strictly smaller than $ d $. 
Since the rule of $ \rightarrow E $ is applied in this derivation, we conclude, according to the natural deduction for \wf , that all assumptions in the derivations $  	\mathcal{D}_{1}$ and $\mathcal{D}_{2}  $ discharged.
So, we replace this by:
$$\vlderivation {
\vliin{}{\rightarrow E}{C}{
\vlhy {{
			\begin{matrix}	
					\mathcal{D} \\
				A_{i}
			\end{matrix}
			}}}
{
\vlhy {{
			\begin{matrix}	
					\mathcal{D}_{i} \\
				A_{i}\rightarrow C
			\end{matrix}
			}}}}  $$
			In this case we reduce the number of cut formulas with rank $ d $. Hence, we conclude that after repeated application of these steps we end up with a normal derivation.
\end{proof}

\begin{thm}
There is an effective procedure for transforming a natural deduction derivation of $ \Gamma\vdash A $ in subintuitionistic logic ${\sf WFI}$ into a normal one also showing  $ \Gamma\vdash A $.
\end{thm}

\begin{proof}
By the proof of Theorem \ref{normal}, we only need to consider the case where we introduce an implication by $\rightarrow_{tr}\!I$ which is then eliminated by $ \elrule{\IMP} $. In this case we are looking at a situation like this
$$\vlderivation {
\vliin{}{\rightarrow E}{ C}{
\vlhy {\vlderivation {
\vliin{}{\rightarrow_{tr}\!I}{A\rightarrow C}{
\vlhy {{
			\begin{matrix}	
					\mathcal{D}_{0} \\
				A\rightarrow B
			\end{matrix}
			}}}
{
\vlhy {{
			\begin{matrix}	
					\mathcal{D}_{1} \\
			B\rightarrow C
			\end{matrix}
			}}}}}}
{
\vlhy {{
				\begin{matrix}
	\\
					 \mathcal{D}_{2}  \\
					A
				\end{matrix} 
			}}}} $$
where the cut rank of $  \mathcal{D}_{2}  $ is strictly smaller than $ d $. 
Since the rule of $ \rightarrow E $ is applied in this derivation, we conclude, according to the natural deduction system for \wf , that all assumptions in the derivations $  	\mathcal{D}_{0}$ and $\mathcal{D}_{1}  $ discharged.
So, we replace this by:

$$\vlderivation {
\vliin{}{\rightarrow E}{ C}{
\vlhy {\vlderivation {
\vliin{}{\rightarrow E}{B}{
\vlhy {{
			\begin{matrix}	
					\mathcal{D}_{2} \\
				A
			\end{matrix}
			}}}
{
\vlhy {{
			\begin{matrix}	
					\mathcal{D}_{0} \\
			A\rightarrow B
			\end{matrix}
			}}}}}}
{
\vlhy {{
				\begin{matrix}
	\\
					 \mathcal{D}_{1}  \\
					B\rightarrow C
				\end{matrix} 
			}}}} $$
			
			In this case we reduce the number of cut formulas with rank $ d $. Hence, we conclude that after repeated application of these steps we end up with a normal derivation.
\end{proof}

\begin{thm}
There is an effective procedure for transforming a natural deduction derivation of $ \Gamma\vdash A $ in subintuitionistic logic ${\sf F}$ into a normal one also showing  $ \Gamma\vdash A $.
\end{thm}
\begin{proof}
The proof of this theorem follows directly from the normalization theorems for the logics 
${\sf WFI}$, ${\sf WFC}$ and ${\sf WFD}$.
\end{proof}

\section{Conclusion}
In this paper, we have, for the first time, introduced natural deduction systems for certain weak subintuitionistic logics and provided proofs of the normalization theorem for each system. A key direction for future research is to explore whether the Curry-Howard correspondence for \ipc can be adapted to these subintuitionistic logics.

\vspace{0.5cm}
\noindent\textbf{Acknowledgements.} I thank Dick de Jongh for all  his helpful suggestions concerning subintuitionistic logics. 

\label{references}
\


\begin{thebibliography}{4}

\bibitem{ab}
N. Aboolian, M. Alizadeh, Lyndon’s interpolation property for the logic of strict implication, \textit{Logic Journal of the IGPLL}, 30(1), 34-70 (2022).

\bibitem{Ar0} M. Ardeshir, \textit{Aspects of Basic Logic}, Ph.D. thesis, Marquette University,
Milwaukee (1995).


\bibitem{Ar1} M. Ardeshir,  W. Ruitenburg, Basic propositional calculus I, \textit{Mathematical Logic
Quarterly} 44, 317-343 (1998).

\bibitem{Ar}
 M. Ardeshir, A translation of intuitionistic predicate logic into basic predicate logic,
\textit{Studia Logica} 62, 341-352 (1999).

\bibitem{Celani}
S. Celani, R. Jansana,  A Closer Look at Some Subintuitionistic Logics, \textit{Notre Dame Journal of Formal Logic}, 42(4), 225-255 (2001).

\bibitem{a4}
 G. Corsi, Weak Logics with strict implication, \textit{Zeitschrift fur Mathematische Logik und Grundlagen der Mathematic}, 33:389-406 (1987). 

 \bibitem{Dic}
D. de Jongh, F. Shirmohammadzadeh Maleki, Subintuitionistic Logics with Kripke Semantics, \textit{In 11th International Tbilisi Symposium on Logic, Language, and Computation, TbiLLC 2015}, LNCS,  pp 333-354, Volume 10148, Springer (2017).

\bibitem{Dic3}
D. de Jongh, F. Shirmohammadzadeh Maleki, Subintuitionistic Logics and the Implications they Prove, \textit{Indagationes Mathematicae}, 10.1016/j.indag.2018.01.013.


\bibitem{Dic4}
D. de Jongh, F. Shirmohammadzadeh Maleki, Two neighborhood Semantics for Subintuitionistic Logics, In \textit{12th International Tbilisi Symposium on Logic, Language, and Computation, TbiLLC 2018}, LNCS, pp 64-85, Volume 11456, Springer (2019).

 \bibitem{FD6}
D. de Jongh, F. Shirmohammadzadeh Maleki, Binary Modal Companions for Subintuitionistic Logics,  \textit{Mathematics, Logic and their Philosophies, Essays in Honour of Mohammad Ardeshir}, pp 35-52 (2021).

\bibitem{Dic5}
D. de Jongh, F. Shirmohammadzadeh Maleki, Binary Modal Logic and Unary Modal Logic, \textit{Logic Journal of the IGPL}, https://doi.org/10.1093/jigpal/jzac083 (2023).

\bibitem{Dosen}
K. Do\v{s}en, Modal Translation in {\sf K} and {\sf D}. In:  \textit{Diamonds and Defaults, Volume 229 of the series Synthese Library}, pp.\ 103-127 (1994).

\bibitem{ish} R. Ishigaki, R. Kashima, Sequent Calculi for Some Strict Implication Logics, \textit{Logic Journal of the IGPL}, 16 (2):155-174 (2008).

\bibitem{kik}K. Kikuchi,  Dual-Context Sequent Calculus and Strict Implication, \textit{Mathematical Logic Quarterly}, 48 (1):87-92 (2002).

\bibitem{a2}
G. Restall, Subintuitionistic Logics,  \textit{Notre Dame Journal of Formal Logic}, Volume 35, Number 1, Winter (1994).

 \bibitem{FD}
F. Shirmohammadzadeh Maleki, D.  de Jongh,  Weak Subintuitionistic Logics, \textit{Logic Journal of the IGPLL}, doi:10.1093/jigpal/jzw062 (2016).

 \bibitem{Fatemeh}
F. Shirmohammadzadeh Maleki, Sequent Calculi for Some Subintuitionistic Logics, arXiv:2410.20425v1 [math.LO] (2024).

 \bibitem{Tr}
A. S. Troelstra, H. Schwichtenberg, \textit{Basic Proof Theory} , 2nd ed., Cambridge University Press (2000).

\bibitem{Tesi} M. Tesi, Subintuitionistic logics and their modal companions: a nested approach, \textit{Journal of Applied Non-Classical Logics}, DOI:10.1080/11663081.2024.2366756 (2024).

\bibitem{a3}
 A. Visser, A propositional logic with explicit fixed points. \textit{Studia Logica}, 40, 2, 198, 155-175 (1981). 

\bibitem{yam}
S. Yamasaki, K. Sano, Constructive Embedding from Extensions of Logics of Strict Implication into Modal Logics,  \textit{Structural Analysis of Non-Classical Logics}, pp 223-251 (2016).

\end{thebibliography}
\end{document}